\documentclass[10pt,a4paper,leqno]{amsart}
\usepackage{amssymb, amsmath, amscd}
\usepackage[dvipsnames]{xcolor}

\def\black{\color{black}}

\usepackage[hypertexnames=false]{hyperref}

\usepackage{color}

\DeclareMathOperator{\Hes}{\rm Hes}

\DeclareMathOperator{\tr}{\rm tr}

\newcommand{\nf}{\nabla f}

\newcommand{\KN}{\mathbin{\bigcirc\mspace{-15mu}\wedge\mspace{3mu}}}
\makeatletter
\@namedef{subjclassname@2020}{%
	\textup{2020} Mathematics Subject Classification}
\makeatother

\newtheorem{theorem}{Theorem}[section]
\newtheorem{lemma}[theorem]{Lemma}

\newtheorem{corollary}[theorem]{Corollary}
\newtheorem{example}[theorem]{Example}

\theoremstyle{definition}
\newtheorem{definition}[theorem]{Definition}

\theoremstyle{remark}
\newtheorem{remark}[theorem]{Remark}

\newcommand\restr[2]{{
  \left.\kern-\nulldelimiterspace 
  #1 
  \vphantom{\big|} 
  \right|_{#2} 
  }}

\begin{document}
\title[Rigidity of weighted Einstein SMMS]{Rigidity of weighted Einstein smooth metric measure spaces}
\author{M. Brozos-V\'azquez, D. Moj\'on-\'Alvarez}
\address{MBV: CITMAga, 15782 Santiago de Compostela, Spain}
\address{\phantom{MBV:}  
	Universidade da Coru\~na, Campus Industrial de Ferrol, Department of Mathematics, 15403 Ferrol,  Spain}
\email{miguel.brozos.vazquez@udc.gal}
\address{DMA: CITMAga, 15782 Santiago de Compostela, Spain}
\address{\phantom{DMA:}
	University of Santiago de Compostela,
15782 Santiago de Compostela, Spain}
\email{diego.mojon@rai.usc.es}
\thanks{Partially supported by projects PID2019-105138GB-C21(AEI/FEDER, Spain) and ED431F 2020/04 (Xunta de Galicia, Spain); and by contract FPU21/01519 (Ministry of Universities, Spain).}
\date{}

\maketitle

\begin{abstract} 
We study the geometric structure of weighted Einstein smooth metric measure spaces with weighted harmonic Weyl tensor. A complete local classification is provided, showing that either the underlying manifold is Einstein, or  decomposes as a warped products in a specific way. Moreover, if the manifold is complete, then it either is a weighted analogue of a space form, or it belongs to a particular family of Einstein warped products.
\end{abstract}


\section{Introduction}\label{sec:intro}

Let $(M,g)$ be a connected Riemannian manifold and $f$ a density function on $M$, which defines the smooth measure $e^{-f}dvol_g$ and gives rise to a smooth metric measure space. A natural problem is to understand how this change of measure should affect the  geometric features of the manifolds when analyzed from a weighted perspective.

Much of this study has been based on the $m$-Bakry-\'Emery Ricci tensor
\begin{equation}\label{eq:Bakry-Emery-Ricci-tensor}
	\rho^m_f=\rho+\operatorname{Hes}_f-\frac{1}{m} df\otimes df,
\end{equation}
where $\rho$ is the usual Ricci tensor of $M$ and $\Hes_f$ refers to the Hessian tensor of the density function (see \cite{Lott} and references therein for some geometric properties of this tensor). Although it was introduced in relation to diffusion processes \cite{Bakry-Emery}, it gave rise to the notion of quasi-Einstein manifolds, which has been extensively studied (see, for example, \cite{Case-Shu-Wei,Catino-GQE,Catino-QE}). Recent works such as \cite{Hua-Wu, wei-wylie} consider bounded Bakry-Émery Ricci tensors to extend gap theorems or to obtain topological restrictions.
 This tensor also appears in Riemannian signature linked to the study of the static perfect fluid Einstein equation \cite{Kobayashi}; and, in the case $m=\infty$, in Perelman's work on the Ricci flow \cite{Perelman}. 
 
 In recent years, however, new {\it weighted} objects have been established, in order to reflect different aspects of the influence of the distinguished measure on the geometry of the associated smooth metric measure space. For example, weighted Yamabe constants are studied in \cite{Case-JDG} in relation to Sobolev inequalities, and weighted analogues of $\sigma_k$-curvatures are introduced in \cite{Case-adv-2016} and further analyzed in \cite{Case-Sigmak}.  In this paper, we aim to further our understanding of smooth metric measure spaces through rigidity results involving the weighted tensors introduced by Case \cite{Case-Tractors,Case-Sigmak}; and to shine light on some fundamental differences between the geometric properties of the usual curvature-related tensors and their weighted counterparts.

Within this framework, a \emph{smooth metric measure space} (\emph{SMMS} for short) is a five-tuple $(M^n,g,f,m,\mu)$, where $(M^n,g)$ is an $n$-dimensional Riemannian manifold (we will consider $n\geq 3$), $f \in C^{\infty}(M)$ is the {\it density function}, which we assume to be  non-constant, $m\in \mathbb{R^+}$ is a dimensional parameter and $\mu\in \mathbb{R}$ is an auxiliary curvature parameter.  We assume $M$ is connected for simplicity, but if it were not then every result should be applied to each connected component. We say that two SMMSs $(M_1^n,g_1,f_1,m_1,\mu_1)$ and $(M_2^n,g_2,f_2,m_2,\mu_2)$ are {\it identified} if there exists an isometry $\psi:(M_1,g_1)\to (M_2,g_2)$ such that $f_1=f_2\circ \psi$, $m_1=m_2$ and $\mu_1=\mu_2$.
In Section~\ref{sec:weighted-geometry}, we delve deeper into the nature of the parameters $m,\mu$ and the meaning of the weighted tensors in the study of SMMSs, particularly as conformal objects. Nevertheless, it is convenient to first establish some of the weighted objects and definitions of interest in our article. These are based (except for the placement of some constants) on those proposed by Case in \cite{Case-Sigmak}. For any SMMS, its {\it weighted scalar curvature} is
\begin{equation}\label{eq:Weighted-scalar-curvature}
	\tau^m_f=\tau+2\Delta f-\frac{m+1}{m}||\nf||^2+m(m-1)\mu \,e^{\frac{2}m f},
\end{equation}
where $\Delta f$ is the Laplacian of $f$ and $||\nf||^2=g(\nf,\nf)$. The {\it weighted Schouten tensor} and {\it weighted Schouten scalar} are given, respectively, by
\begin{equation}\label{eq:Weighted-Schouten-tensor}
	P^m_f=\frac{1}{n+m-2}(\rho^m_f-J^m_f g), \qquad J^m_f=\frac{1}{2(n+m-1)}\tau^m_f.
\end{equation}
Although $\tau_f^m$ is regarded as a weighted analogue of the usual scalar curvature $\tau$, it is not the trace of the Bakry-\'Emery Ricci tensor \eqref{eq:Bakry-Emery-Ricci-tensor}. Moreover, $J^m_f$ is not the trace of $P^m_f$, as opposed to the usual Schouten tensor $P=\frac{1}{n-2}\left(\rho-J g\right)$, where $J=\operatorname{Tr}P=\frac{\tau}{2(n-1)}$. The difference between these two quantities is denoted by $Y^m_f=J^m_f-\tr P^m_f$ and will also play a role. 

Following this pattern of generalization of Riemannian objects to the weighted setting, a natural step is finding a suitable analogue to Einstein manifolds, which gives rise to the notion of weighted Einstein manifolds.

\begin{definition} \cite{Case-Sigmak}
	A {\it weighted Einstein manifold} is a smooth metric measure space $(M^n,g,f,m,\mu)$ such that $P^m_f=\lambda g$ for some $\lambda\in \mathbb{R}$.
\end{definition}

Weighted Einstein manifolds are the main focus of this work. We will see in Theorem~\ref{th:analytic_solutions} that the underlying metric and the density function of weighted Einstein SMMSs are real analytic in harmonic coordinates. Moreover, they are particular cases of {\it generalized quasi-Einstein (GQE) manifolds} \cite{Catino-GQE}, i.e. four-tuples $(M^n,g,f,m)$ such that $\rho_f^m=\alpha g$ for some $\alpha\in C^\infty(M)$. Indeed, a weighted Einstein manifold is GQE with 
	\begin{equation}\label{eq:constant-wE}
	\alpha=(n+m-2)\lambda+J_f^m=(2n+m-2)\lambda+Y_f^m.
	\end{equation}
	Generalized quasi-Einstein manifolds have been extensively studied in literature. For example, under conditions such as the harmonicity of the Weyl tensor, classification results for GQE manifolds have been found in both Riemannian \cite{Catino-GQE} and Lorentzian \cite{Brozos-Iso-GQE} signature. However, we aim to further our understanding of SMMSs by imposing appropriate conditions on the weighted tensors themselves, rather than on their unweighted analogues; and studying how they affect the geometry of the underlying manifold, not only as a SMMS, but also as a purely Riemannian object. 
	
	To that end, we consider the {\it weighted Weyl tensor}
\begin{equation}\label{eq:Weighted-Weyl-tensor}
	W^m_f=R-P^m_f\KN g,
\end{equation}
where $\KN$ denotes the Kulkarni-Nomizu product. We transfer the harmonicity condition of the Weyl tensor used in \cite{Brozos-Iso-GQE,Catino-GQE} to our weighted setting as the {\it weighted harmonicity} of the weighted Weyl tensor, $0=\delta_f W^m_f=\delta W^m_f-\iota_{\nf} W^m_f$. Here, $\delta$ is the usual divergence, and $\iota$ stands for an interior product (see Section~\ref{sec:weighted-geometry} for details). For brevity, when a SMMS satisfies $\delta_f W^m_f=0$, we will say that it has {\it weighted harmonic Weyl tensor}. Finally, we will also need the {\it weighted Cotton tensor}
\begin{equation}\label{eq:Weighted-Cotton-tensor}
	dP^m_f(X,Y,Z)=(\nabla_XP^m_f)(Y,Z)-(\nabla_YP^m_f)(X,Z).
\end{equation}

We will thus focus on finding rigidity results for weighted Einstein manifolds such that $\delta_f W^m_f=0$. In the unweighted context, when working with Einstein manifolds, the condition $\delta W=0$ does not provide any additional information, since  all Einstein manifolds have harmonic Weyl tensor (moreover, they have harmonic curvature tensor). Nevertheless, for SMMS, the condition $\delta_f W^m_f=0$ does not follow from $P^m_f=\lambda g$, (see Examples~\ref{ex:wenonharm1} and \ref{ex:wenonharm2}). Therefore, it is pertinent to wonder to what extent this weighted harmonicity condition restricts the geometry of weighted Einstein manifolds.

\subsection{Main results}

Our aim is to characterize the geometric structure of weighted Einstein manifolds with weighted harmonic Weyl tensor. We will consider manifolds of dimension $n\geq 3$, and arbitrary values for the dimensional and curvature parameters $m$ and $\mu$. The following family of examples will play a role, as we will see subsequently. 

\begin{example}\label{ex:not-Einstein}\rm
	Take a SMMS of the form $(I\times_\varphi N,g,f,\frac{1}2,0)$, where $I\times_\varphi N$ is a warped product of an open interval $I\subset \mathbb{R}^+$ and a Ricci flat manifold $N$. Now, set the warping and density functions
	\[\varphi(t)=A(Bt)^{\frac{1}{n-1}}, \quad f(t)=-\log(Bt),\]
	where $t$ is the natural coordinate in $\mathbb{R}^+$ and $A,B\in \mathbb{R}^+$. The resulting  examples of SMMSs  satisfy $P^{1/2}_f=0$ and $\delta_fW_f^{1/2}=0$,  so they are weighted Einstein and have weighted harmonic Weyl tensor. However, the scalar curvature of the underlying manifold is non-constant, $\tau=\frac{(n-2)}{(n-1)t^2}$, therefore they are not Einstein. 
	
	In particular, if $N$ is the usual flat Euclidean space $\mathbb{R}^{n-1}$  with coordinates $(x_1,\dots,x_{n-1})$, the weighted Einstein tensor $W_f^{1/2}$ presents the following non-zero components (up to symmetries):
		\[
			\begin{array}{rcl}
 				W_f^{1/2}(\partial_t,\partial_{x_i},\partial_t,\partial_{x_i})&=&\frac{ (n-2)\varphi(t)^2}{(n-1)^2t^2},
 				\\ \noalign{\medskip}
				 \quad W_f^{1/2}(\partial_{x_i},\partial_{x_j},\partial_{x_i},\partial_{x_j})&=&-\frac{\varphi(t)^4}{(n-1)^2t^2}
				 , \quad i\neq j.
			\end{array}
	\]
	Note that $(I\times_\varphi N,g)$ is an incomplete manifold, and it cannot be isometrically embedded in any complete manifold (see Lemma~\ref{lemma:Ricci-blowup}).
\end{example}


The relevance of Example~\ref{ex:not-Einstein} becomes clear from the following theorem, which states that, around regular points of $f$, any non-Einstein weighted Einstein SMMS with weighted harmonic Weyl tensor is given by this example. Moreover, since $f$ is real analytic in harmonic coordinates, the set of its regular points is dense in $M$ (see Remark~\ref{remark:open-dense}). Thus, the following result actually determines the local geometric features of an open dense subset of these SMMSs.

\begin{theorem}\label{th:main-local-result}
	Let $(M^n,g,f,m,\mu)$ be a SMMS such that $P_f^m=\lambda g$ and $\delta_fW_f^m=0$. Then, for each regular point $p$ of $f$, there exists a neighborhood $\mathcal{U}$ of $p$ which is isometric to a warped product $I\times_\varphi N$, where $I\subset \mathbb{R}$ is an open interval, $N$ is an $(n-1)$-dimensional Einstein manifold, and $\nf$ is tangent to $I$. Moreover, one of the following conditions holds:
	\begin{enumerate}
		\item{ $I\times_\varphi N$ is Einstein with $\rho=2(n-1)\lambda g$.}
		\medskip
		\item{$(\mathcal{U},\restr{g}{U},\restr{f}{U},m,\mu)$ is identified with $(I\times_\varphi N,g,f,\frac{1}2,0)$ as given in Example~\ref{ex:not-Einstein}.}
	\end{enumerate}
\end{theorem}

\begin{remark}\label{remark:harmonic-Weyl}
Warped products of the form $I\times_\varphi N$ with $N$ Einstein have harmonic Weyl tensor  (see \cite[16.26(i)]{Besse}  and \cite{Gebarowski}). Hence, it follows from the local structure described in Theorem~\ref{th:main-local-result} and the density of regular points of $f$ that the underlying manifold of a SMMS which is weighted Einstein and has weighted harmonic Weyl tensor is not necessarily Einstein, but does have harmonic Weyl tensor in the unweighted sense.
\end{remark}

The Einstein case in Theorem~\ref{th:main-local-result} presents a fair level of rigidity and we are able to achieve a local classification around regular points of $f$ (see Section~\ref{sec:Einstein}). Unlike in the non-Einstein scenario, it is shown in Theorem~\ref{th:Einstein-case} that there exist local examples for all values of $\lambda$, $m$ and $\mu$.

Theorem~\ref{th:main-local-result} is local in nature. However, if the manifold is complete, the only examples are the weighted analogues of space forms (see Section~\ref{sec:weighted-geometry}) or a particular family of warped products with $\lambda<0$. This fact is stated in the following rigidity result.

\begin{theorem}\label{th:complete-global}
	Let $(M^n,g,f,m,\mu)$ be a complete SMMS with $P_f^m=\lambda g$ and $\delta_fW_f^m=0$. Then, $(M^n,g,f,m,\mu)$ is identified with one of the following spaces attending to the sign of $\lambda$:
	\begin{enumerate}
		\item{ \underline{$\lambda>0$:} the $m$-weighted $n$-sphere of constant sectional curvature $2\lambda$ (Example~\ref{ex:sphere}).}
		\smallskip
		\item{ \underline{$\lambda=0$:} the $m$-weighted $n$-Euclidean space (Example~\ref{ex:euclidean-space}).}
		\smallskip
		\item{ \underline{$\lambda<0$:}
		\smallskip
		\begin{enumerate}
			\item{the $m$-weighted $n$-hyperbolic space of constant sectional curvature $2\lambda$ (Example~\ref{ex:hyperbolic-space}), or}
			\smallskip
			\item{$(M,g)$ is an Einstein warped product $\mathbb{R}\times_\varphi N$, where $N$ is a Ricci flat complete manifold. In this case, there is a coordinate $t$ parameterizing $\mathbb{R}$ by arc length such that the warping and density functions take the forms
	\begin{center}
	$\varphi (t)= A e^{\sqrt{-2\lambda}\,t}, \qquad f(t)=-m\log\left(B+ AC( e^{\sqrt{-2\lambda}\,t}-1)\right)$,
	\end{center}
	 with $\mu=-2(B-AC)^2\lambda$ or $m=1$, for $A,B,C\in \mathbb{R}^+$ such that $AC\leq B$.}
		\end{enumerate}}
	\end{enumerate}	
\end{theorem}

Note that we could eliminate the logarithm in the density function in Theorems~\ref{th:main-local-result} and \ref{th:complete-global} by the change of variable $v=e^{-f/m}$. The choice of $v$ as a density is related to the interpretation of weighted objects as  their corresponding standard Riemannian  counterparts on certain warped products (we refer to Section~\ref{sec:weighted-geometry} and Remark~\ref{remark:multi-warped} for details).

\subsection{Outline of the paper}
The remaining of this paper is organized as follows.

In Section~\ref{sec:weighted-geometry}, we go over some preliminaries on the geometric significance of the weighted tensors we have defined in the introduction; and include an analyticity result for weighted Einstein manifolds. 

In Section~\ref{sec:WEWH-manifolds}, we begin by computing some geometric formulas and proving the local splitting of these SMMSs as warped products with Einstein fiber. Afterwards, we find necessary and sufficient conditions, in terms of an overdetermined system of ODEs, for a SMMS to satisfy both the weighted Einstein and the weighted harmonicity conditions. We use these ODEs to prove Theorem~\ref{th:main-local-result}. 

In Section~\ref{sec:Einstein}, we focus on the Einstein case that arises in Theorem~\ref{th:main-local-result}--(1). We describe both the warping and density functions $\varphi$ and $f$, as well as the Einstein constant of the fiber and the value of the parameter $\mu$, to obtain Theorem~\ref{th:Einstein-case}, which completes a classification result around regular points of $f$. In low dimensions, stronger rigidity results are provided in Corollary~\ref{cor:4-dim-curv}.

Finally, in Section~\ref{sec:global-results}, we prove Theorem~\ref{th:complete-global}, determining the only four families of complete weighted Einstein SMMSs with weighted harmonic Weyl tensor. 
 
\medskip

 The authors gratefully acknowledge valuable comments from J. Case.

\section{Weighted geometry in smooth metric measure spaces}\label{sec:weighted-geometry}

Before moving on with the analysis of weighted curvature-related tensors, it is convenient to fix notation and introduce all needed objects about SMMSs. On a Riemannian manifold $(M,g)$, let  $R(X,Y)=\nabla_{[X,Y]}-[\nabla_X,\nabla_Y]$ be the curvature tensor and define the Kulkarni-Nomizu product of two symmetric $(0,2)$-tensors, $T$ and $S$, as
\[
\begin{array}{rcl}
 	(T\KN S)(X,Y,Z,U)&=&\,T(X,Z)S(Y,U)+T(Y,U)S(X,Z) \\
	\noalign{\smallskip}
	&& -T(X,U)S(Y,Z)-T(Y,Z)S(X,U),
\end{array}
\]
where $X$, $Y$, $Z$ and $U$ are arbitrary vector fields. Let $(M^n,g,f,m,\mu)$ be a SMMS, i.e.,  an $n$-dimensional Riemannian manifold $(M^n,g)$ ($n\geq 3$) endowed with a non-constant density function $f \in C^{\infty}(M)$, a dimensional parameter $m\in \mathbb{R^+}$ and an auxiliary curvature parameter $\mu\in \mathbb{R}$. As mentioned in Section~\ref{sec:intro}, the geometry of a SMMS is best described in terms of {\it weighted objects}, which incorporate information about the density function while retaining suitable geometric meanings. For example, the {\it weighted volume element} $e^{-f}dvol_g$ represents a change in the way we measure volumes on the manifold. We also define the {\it weighted divergence} of a $k$-covariant tensor $T$ as $\delta_f T=\delta T-\iota_{\nf} T$, where $\delta$ is the usual divergence, given in an orthonormal frame $\{E_1,\dots,E_n\}$ by $\delta T(\cdots)=\sum_{i=1}^n(\nabla_{E_i}T)(E_i,\cdots)$,
and $\iota_{\nf}$ refers to the interior product $\iota_{\nf} T(\cdots)=T(\nf,\cdots)$.
This definition is motivated by the fact that this operator is (up to a sign) the formal adjoint of the exterior derivative on $k$-forms with respect to $e^{-f}dvol_g$ (see \cite{Case-Tractors}). We say that $T$ is {\it weighted harmonic if $\delta_f T=0$}.

By defining $v=e^{-f/m}$, one can equivalently describe a SMMS as a four-tuple $(M^n,g,v^mdvol_g, \mu)$, introducing the dimensional parameter $m$ implicitly into the distinguished measure. In the special case $m=0$, by convention, $f$ is assumed to be identically zero, so that SMMSs with $m=0$ reduce to usual Riemannian manifolds. In this work, we focus on proper SMMSs, hence why we only consider those with $m\in \mathbb{R}^+$. 

Many of the weighted objects defined in  Section~\ref{sec:intro} can be thought of as restrictions of the usual Riemannian features of the formal warped product
\begin{equation}\label{eq:formal-warped-product}
	M^n\times_v F^m(\mu)=(M^n\times F^m, g\oplus v^2 h (\mu)),
\end{equation}
where $(F^m,h (\mu))$ is the $m$-dimensional space form of sectional curvature $\mu$. Indeed, when the warped product makes sense, its scalar curvature is the weighted scalar curvature \eqref{eq:Weighted-scalar-curvature}; the weighted volume element is the restriction of its Riemannian volume element to $M$; and its Ricci tensor restricts to vectors tangent to $M$ as the $m$-Bakry-Émery Ricci tensor \eqref{eq:Bakry-Emery-Ricci-tensor}. Note that, for $m=1$, the fiber is one-dimensional and thus has no sectional curvature, so the auxiliary parameter $\mu$ becomes irrelevant. Since $\mu$ is not needed, SMMSs with $m=1$ can be denoted by a four-tuple  $(M^n,g,f,1)$.

On the other hand, the weighted Schouten and Weyl tensors are introduced in order to discuss the conformal properties of SMMS. Thus, a notion of conformal class in the weighted sense is needed. Following \cite{Case-Sigmak}, the SMMSs $(M^n,g,f,m,\mu)$ and $(M^n,\tilde{g},\tilde{f},m,\mu)$ are {\it pointwise conformally equivalent} if there exists a smooth function $u\in C^\infty(M)$ such that $\tilde{g}=e^{-2u/m}g$ and $\tilde{f}=f+ u$. In particular, a SMMS $(M^n,g,f,m,\mu)$ with $m\neq 1$ is {\it locally conformally flat in the weighted sense} if for each point there is a neighborhood so that the restriction of the SMMS is pointwise conformally equivalent to an open set of $(F^n,h(-\mu),0,m,\mu)$. In the distinguished case $m=1$, we say that $(M^n,g,f,1)$ is {\it locally conformally flat in the weighted sense} if for each point there exists a neighborhood so that the restriction of the SMMS is pointwise conformally equivalent to $(F^n,h(c),0,1)$ for any sectional curvature $c$. Equivalently, $(M^n,g,f,m,\mu)$ is locally conformally flat if the formal warped product \eqref{eq:formal-warped-product} is locally conformally flat.

The weighted Weyl tensor $W^m_f$ is a weighted conformal invariant of  a SMMS (see \cite{Case-Tractors}). Hence, in a weighted conformal class, the Riemann curvature tensor $R$ is controlled by the weighted Schouten tensor. Due to this fact, the weighted Weyl tensor is intimately related to the notion of local conformal flatness in the weighted sense. Indeed, a SMMS $(M^n,g,f,m,\mu)$ with $n\geq 3$ is locally conformally flat in the weighted sense if and only if  $W_f^m=0$. In this case, we also have $dP_f^m=0$, since both tensors are related by the weighted divergence in a similar (but not identical) manner to how the unweighted Weyl and Cotton tensors are related by the usual divergence. We refer to \cite{Case-Tractors,Case-Sigmak} for further details on the treatment of SMMS as conformal objects.

Going back to the main focus of our work, weighted Einstein manifolds are intimately related to quasi-Einstein ones, through a series of tensorial equations (see \cite{Case-Sigmak}) which give rise to the following key lemma.

\begin{lemma} \cite{Case-Sigmak} \label{lemma:scale}
	Let  $(M^n,g,f,m,\mu)$ be a SMMS such that $P^m_f=\lambda g$  for some $\lambda\in \mathbb{R}$. Then, there is a unique constant $\kappa\in \mathbb{R}$ (called {\it scale}) such that $J_f^m=(m+n)\lambda -m\kappa e^{\frac{f}m}$. Moreover, $(M^n,g,f,m)$ is quasi-Einstein (i.e. $\rho_f^m=\alpha g$ for some $\alpha\in \mathbb{R}$) if and only if, for an appropriate $\mu$, $(M^n,g,f,m,\mu)$ is weighted Einstein with $\kappa =0$.
\end{lemma}

On the other hand, for a weighted Einstein manifold with $P_f^m=\lambda g$, the weighted Weyl tensor $W^m_f$ defined in \eqref{eq:Weighted-Weyl-tensor} takes the form
\begin{equation}\label{eq:weighted-weyl-Einstein}
	W_f^m(X,Y,Z,U)=R(X,Y,Z,U)-2\lambda\{g(X,Z)g(Y,U)-g(X,U)g(Y,Z)\}.
\end{equation}
Hence, it follows that for a SMMS that is weighted Einstein and locally conformally flat in the weighted sense, the underlying Riemannian manifold has constant sectional curvature $2\lambda$. Indeed, the simplest examples of weighted Einstein manifolds, which will be essential in our discussion of complete SMMSs, come from considering weighted analogues to the usual model space forms as follows. 

%

\begin{example}\rm\label{ex:sphere} 
	For $\lambda>0$, consider the $n$-sphere $S^n(2\lambda)$ of constant sectional curvature $2\lambda$ (equivalently, of radius $\frac{1}{\sqrt{2\lambda}}$). Take the usual round metric, written as a warped product with fiber the unitary sphere $S^{n-1}$:
	\[
		g_S^{2\lambda}=dt^2+(2\lambda)^{-1}\sin^2( \sqrt{2\lambda} t) g_{S^{n-1}}, \quad t\in \left(0,\frac{\pi}{\sqrt{2\lambda}}\right),
	\]
	which extends smoothly to the two poles of the sphere, $N$ and $-N$, where $t$ measures the geodesic distance from $N$.  Take $f_m(t)=-m\log (A+B\cos(\sqrt{2\lambda} t))$ for $A\in \mathbb{R}^+$, $B\in \mathbb{R}^*$ such that $A>|B|$, so that $v(t)=A+B\cos(\sqrt{2\lambda} t)$ is positive for all $t\in \left[0,\frac{\pi}{\sqrt{2\lambda}}\right]$ (note that, by continuity, $v(N)=A+B$ and $v(-N)=A-B$). For $m\neq 1$, fix $\mu=2\lambda(B^2-A^2)$. Then, $(S^n(2\lambda),g^{2\lambda}_S,f_m,m,\mu)$ and $(S^n(2\lambda),g^{2\lambda}_S,f_1,1)$ are complete SMMSs satisfying $P_f^m=\lambda g$ and have vanishing weighted Weyl tensor. We call this the {\it $m$-weighted $n$-sphere}. The scale of this SMMS is $\kappa=2\lambda A$.
	
	Note that, by removing the condition $A>|B|$, we also get incomplete examples defined on the open set of points where $v>0$. We present two notable particular examples (cf. \cite{Case-Sigmak}):
	\begin{enumerate}
		\item $A=B=1$ on the punctured sphere $S^n(2\lambda)-\{-N\}$ gives the {\it standard $m$-weighted n-sphere} of curvature $2\lambda$.
		\item $A=0$ and $B=1$ on the upper hemisphere $S^n_+(2\lambda)$ gives the {\it positive elliptic $m-Gaussian$}, which is a quasi-Einstein manifold since its scale vanishes (indeed, $J_f^m=\lambda(m+n)$ in this case).
	\end{enumerate}  
\end{example}

\begin{example}\rm\label{ex:euclidean-space} 
	For $\lambda=0$, consider the Euclidean $n$-space $\mathbb{R}^+$, given in spherical coordinates by the metric
	\[
	g_E=dt^2+t^2 g_{S^{n-1}}, \quad t\in (0,\infty),
	\]
	which extends smoothly to $t=0$. Take $f_m(t)=-m\log (A+Bt^2)$ for $A,B\in \mathbb{R}^+$. For $m\neq 1$, fix $\mu=-4AB$. Then, $(\mathbb{R}^n,g_E,f_m,m,\mu)$ and $(\mathbb{R}^n,g_E,f_1,1)$ are complete SMMSs satisfying $P_f^m= 0$ and have vanishing weighted Weyl tensor. We call this the {\it $m$-weighted $n$-Euclidean space}. The scale of this  weighted Einstein SMMS is $\kappa=2 B$ and there are no nontrivial (i.e. with $f$ not constant) quasi-Einstein manifolds of this type.
\end{example}

\begin{example}\rm\label{ex:hyperbolic-space} 
	For $\lambda<0$, consider the model of hyperbolic $n$-space of constant sectional curvature $2\lambda$, denoted by $H^n(2\lambda)$, given by the metric
	\[
		g^{2\lambda}_H=dt^2+(-2\lambda)^{-1}\sinh^2( \sqrt{-2\lambda} t) g_{S^{n-1}}, \quad t\in (0,\infty),
	\]
	which extends smoothly to $t=0$. Take $f_m(t)=-m\log (A+B\cosh(\sqrt{-2\lambda} t))$ for $B\in \mathbb{R}^+$, $A\in \mathbb{R}$ such that $A>-B$, so that $v(t)=A+B\cosh(\sqrt{-2\lambda} t)$ is positive for all $t\in [0,\infty)$. For $m\neq 1$, fix $\mu=2\lambda (B^2-A^2)$. Then, $(H^n(2\lambda),g^{2\lambda}_H,f_m,m,\mu)$ and $(H^n(2\lambda),g^{2\lambda}_H,f_1,1)$ are complete SMMSs satisfying $P_f^m=\lambda g$ and have vanishing weighted Weyl tensor. We call this the {\it $m$-weighted $n$-hyperbolic space}. The scale of this SMMS is $\kappa=2\lambda A$, so it is quasi-Einstein when $A=0$. Note that, in contrast to Example~\ref{ex:sphere}, taking $A=0$ results in a complete manifold. 
\end{example}


\subsection{Analyticity of  weighted Einstein manifolds.}

The analytic properties of the weighted Einstein equation allow for a further understanding of the geometric properties of the underlying Riemannian manifold of the SMMS. Indeed, by similar arguments to those used for Einstein and quasi-Einstein metrics (see \cite[5.26]{Besse} and \cite{He-warped-Einstein}, respectively), we can prove the real analyticity, in harmonic coordinates, of both the metric and the density function.

\begin{theorem}\label{th:analytic_solutions}
	Let $(M^n,g,f,m,\mu)$ be a weighted Einstein SMMS. Then, both $g$ and $f$ are real analytic in harmonic coordinates on $M$.
\end{theorem}
\begin{proof}
 Assume $P_f^m=\lambda g$ and take traces in this equation to obtain
\begin{multline*}
	(n+2m-2)\tau+2(m-1)\Delta f +\frac{(m-1)(n-2)}{m}||\nf||^2\\  - nm(m-1)\mu e^{\frac{2}mf}=2(n+m-1)(n+m-2)n\lambda.
\end{multline*}
Note that, if $m=1$, this becomes $\tau=2n(n-1)\lambda$, so $\tau$ is constant. If $m\neq 1$, we can write
	\begin{equation}\label{eq:laplacian_1}
			\begin{array}{rcll}
 				\Delta f&=&\frac{1}{2(m-1)}(-(n+2m-2)\tau+2(n+m-1)(n+m-2)n\lambda) \\
				\noalign{\medskip}
				&&-\frac{n-2}{2m}||\nf||^2+\frac{ nm}{2}\mu e^{\frac{2}mf}.
			\end{array}
\end{equation}
On the other hand, let $\kappa\in \mathbb{R}$ be the unique constant such that $J_f^m+m\kappa e^{\frac{f}m}=(m+n)\lambda$ (see Lemma~\ref{lemma:scale}). Solving this equation for $\tau$ yields
\begin{equation}\label{eq:tau_kappa}
			\begin{array}{rcll}
 				\tau+ 2\Delta f&=&2(n+m-1)\left((m+n)\lambda-m\kappa e^{\frac{f}m}\right) \\
				\noalign{\medskip}
				&&+\frac{m+1}{m}||\nf||^2-m(m-1)\mu \,e^{\frac{2}m f}.
			\end{array}
\end{equation}
 If $m=1$, since $\tau$ is constant, equation \eqref{eq:tau_kappa} becomes
	$\Delta f+\, \mathrm{l.o.t}=0$,
where l.o.t. stands for lower order terms involving the metric and the density function.  If $m\neq 1$,  we can use \eqref{eq:tau_kappa} to write $\tau+2\Delta f+\, \mathrm{l.o.t}=0$, while, by \eqref{eq:laplacian_1}, we have $\Delta f+\frac{n+2m-2}{2(m-1)}\tau+\, \mathrm{l.o.t}=0$.
Combining both equations, we have	
	\[
		\frac{n+m-1}{m-1}\Delta f+\, \mathrm{l.o.t}=0.
	\]
Since $n\geq 3$ and $m\in \mathbb{R}^+-\{1\}$, we can write $\Delta f+\, \mathrm{l.o.t}=0$. On the other hand, the weighted Einstein equation $\rho_f^m=((n+m-2)\lambda+J_f^m) g$ takes the form
\[
	\rho+\operatorname{Hes}_f-\frac{1}{m} df\otimes df=(2(n+m-1)\lambda-m\kappa e^{\frac{f}m}) g
\]
so, for any $m\in \mathbb{R}^+$, we end up with
\[
		\begin{array}{rcll}
 				\rho+\operatorname{Hes}_f+\, \mathrm{l.o.t}&=&0, \\
				\noalign{\medskip}
				\Delta f+\, \mathrm{l.o.t}&=&0.
		\end{array}
\]
In harmonic coordinates, these geometric equations become a quasi-linear  second-order system of EDPs:
\[
		\begin{array}{rcll}
 				-\frac{1}{2}g^{rs}\frac{\partial^2 g_{ij}}{\partial x^r \partial x^s}+\frac{\partial^2 f}{\partial x^i \partial x^j}+\, \mathrm{l.o.t}&=&0, \\
				\noalign{\bigskip}
				g^{rs}\frac{\partial^2 f}{\partial x^r \partial x^s}+\, \mathrm{l.o.t}&=&0.
		\end{array}
\]
The principal symbol $\sigma_\xi:S^2T^*\!M\oplus C^\infty(M)\rightarrow S^2T^*\!M\oplus C^\infty(M)$ is given by
\[
	(h,\omega)\mapsto \sigma_\xi(h,\omega)=\left(-\frac{1}2||\xi||^2 h+ \omega \xi\otimes \xi,||\xi||^2\omega\right).
\]
If $\sigma_\xi(h,\omega)=0$ and $\xi\neq 0$, then it follows that $\omega=0$. In this case, $h$ must also vanish. Thus, $\sigma_\xi$ is an automorphism of $S^2T^*M\oplus C^\infty(M)$, and the quasi-linear system is elliptic. Moreover, the whole system is of the form $F(g,f,\partial g, \partial f, \partial^2 g, \partial^2 f)=0$, where $F$ is real analytic. It follows that both the metric and the density function are real analytic in harmonic coordinates (see \cite[J.41]{Besse}).
\end{proof}

\begin{remark}\label{remark:open-dense}
Note that, since  $f$ is non-constant and, by Theorem~\ref{th:analytic_solutions}, it is real analytic in harmonic coordinates, the set $\tilde{M}=\{p\in M| (\nf)_p\neq 0\}$ of regular points of $f$ is open and dense in $M$. This fact will be crucial in the discussion of the global structure of complete weighted Einstein SMMSs in Section~\ref{sec:global-results}.
\end{remark}

\section{Weighted Einstein manifolds with weighted harmonic Weyl tensor}\label{sec:WEWH-manifolds}

The objective in this section is to study the local geometric structure of weighted Einstein manifolds with weighted harmonic Weyl tensor in any dimension, for arbitrary values of the parameters $m$ and $\mu$. First, we shall point out that, in contrast to the unweighted setting, where the Einstein condition implies the harmonicity of the Weyl tensor, the condition $\delta_f W_f^m=0$ does not follow from $P_f^m=\lambda g$, for any value of the dimensional parameter $m$. The following examples illustrate this fact appropriately.

\begin{example}\label{ex:wenonharm1}\rm
	Let $m\in \mathbb{R}^+-\{\frac{1}2\}$, and let $(\mathbb{R}^+\times\mathbb{R}^3,g)$ be the $4$-dimensional Riemannian manifold with local coordinates $(x_1,\dots,x_4)$ and metric given by the only non-vanishing components $g(\partial_{x_i},\partial_{x_i})=x_1^{2m}$. For the density function  $ f_m(x_1)=-2m(m+1)\log(x_1)$, and the parameter $\mu=0$, this manifold satisfies $P_f^m=0$. However, its weighted Weyl tensor is not weighted harmonic. The non-vanishing components of $\delta_fW^{m}_f$ are $\delta_fW^{m}_f(\partial_{x_i},\partial_{x_1},\partial_{x_i})=\frac{2m(2m^2+m-1)}{x_1^3}$, $i>1$ (up to symmetries). 
	
Notice that the value $m=\frac12$ is special in this family. Indeed, the SMMS $(\mathbb{R}^+\times\mathbb{R}^3,g,f_{1/2},\frac{1}2,0)$ corresponds to Example~\ref{ex:not-Einstein} for $t=\frac23 (x_1)^{3/2}$, $A=1$ and $B=\frac32$. The weighted Schouten tensor and the weighted divergence of its weighted Weyl tensor vanish ($P_f^{1/2}=0$ and $\delta_fW^{1/2}_f=0$), but the weighted Weyl tensor itself does not.
\end{example}

The following SMMS in dimension 3 provides an example of weighted Einstein manifold with $\delta_f^mW_f^m\neq 0$ for the remaining value $m=\frac{1}2$.

\begin{example}\label{ex:wenonharm2}\rm
	Let $(\mathbb{R}^+\times \mathbb{R}^2,g)$ be the 3-dimensional Riemannian manifold with the metric given by the only non-vanishing components $g(\partial_{x_i},\partial_{x_i})=x_1^{\frac{2}3(3-\sqrt{6})}$. For the density function $f(x_1)=-\sqrt{\frac{2}3}\log(x_1)$, and the parameters $m=\frac{1}2$ and $\mu=0$, this manifold satisfies $P_f^{1/2}=0$. However, its weighted Weyl tensor is not harmonic in the weighted sense. The non-vanishing components of $\delta_{f}W^{1/2}_{f}$ are $\delta_fW^{1/2}_f(\partial_{x_i},\partial_{x_1},\partial_{x_i})=\frac{4(\sqrt{6}-3)}{9x_1^3}$, $i>1$ (up to symmetries).
\end{example}

\begin{remark}\rm\label{remark:counterexamples-warped}
Although several geometrical conditions on SMMSs have a counterpart on formal warped products \eqref{eq:formal-warped-product}, the conditions we are considering in this work do not in general. In fact, fix $m=3$ in Example~\ref{ex:wenonharm1} and consider the warped product  $\mathbb{R}^+\times\mathbb{R}^3\times_{v} \mathbb{R}^3$, where $v(x_1)=x_1^8$. This warped product is not Einstein but has harmonic Weyl tensor. Thus, Example~\ref{ex:wenonharm1} additionally illustrates that a weighted Einstein SMMS does not give rise to an Einstein warped product \eqref{eq:formal-warped-product} and that a warped product \eqref{eq:formal-warped-product} does not induce a SMMS with weighted harmonic Weyl tensor.

Conversely, for $\mu\neq 0$, consider the SMMS $(\mathbb{R}^+\times_\varphi\mathbb{R}^3,f,2,\mu)$, where  $\varphi(t)=t^{\frac{1}{3}}$ and $f(t)=-\log(t)$ (hence $v(t)=t^{\frac{1}2}$), as in Example~\ref{ex:not-Einstein}. This SMMS has weighted harmonic Weyl tensor, although it is not weighted Einstein for any $\lambda\in \mathbb{R}$. However, the warped product $\mathbb{R}^+\times_\varphi\mathbb{R}^3\times_{v} F^2(\mu)$ does not have harmonic Weyl tensor. Thus, the weighted harmonicity condition on the weighted Weyl tensor does not induce, in general, a warped product \eqref{eq:formal-warped-product} with harmonic Weyl tensor (cf. Remark~\ref{remark:multi-warped}).
\end{remark}

We begin the analysis of weighted Einstein SMMSs with the following lemma, which shows a useful way of expressing the weighted divergence $\delta_fW_f^m$, where the function $Y^m_f=J^m_f-\tr P^m_f$ is involved. 

\begin{lemma}\label{lemma:Main-Expression-1}
	Let $(M^n,g,f,m,\mu)$ be a SMMS such that $P_f^m=\lambda g$. Then, the following equation is satisfied for all $X,Y,Z\in \mathfrak{X}(M)$:
	\begin{equation}\label{eq:Main-Expression-1}
\begin{array}{rcl}
 	\delta_fW_f^m(X,Y,Z)&=&\left(\frac{1}{m}Y^m_f+\lambda \right)\left\{df(Y)g(X,Z)-df(Z)g(X,Y)\right\}\\
	\noalign{\smallskip}
	&& -\frac{1}{m}\left\{df(Y)\Hes_f(X,Z)-df(Z)\Hes_f(X,Y) \right\}.
\end{array}
	\end{equation}
\end{lemma}
\begin{proof}
Firstly, we recall the following relation for arbitrary SMMSs (see \cite{Case-Tractors} for a detailed derivation, or \cite{Case-Sigmak} for the form used here):
\[
	\mathrm{tr}(dP^m_f)=\iota_{\nf}P^m_f+dY_f^m-\frac{1}{m}Y_f^mdf.
\]
Since the weighted Schouten tensor satisfies $P_f^m=\lambda g$, it is Codazzi, so the weighted Cotton tensor \eqref{eq:Weighted-Cotton-tensor} vanishes and the expression above transforms into $dJ_f^m=dY^m_f=\left(\frac{1}{m}Y^m_f-\lambda \right)df$.  On the other hand, we calculate the covariant derivative of the Bakry-\'Emery Ricci tensor,
\[
\begin{array}{rcl}
 	(\nabla_Y\rho^m_f)(X,Z)&=&(\nabla_Y\rho)(X,Z)+g(\nabla_Y\nabla_Z\nf,X)-g(\nabla_{\nabla_YZ}\nf,X) \\
	\noalign{\smallskip}
	&& -\frac{1}m\left\{df(X)\Hes_f(Y,Z)+df(Z)\Hes_f(X,Y)\right\}. \\
\end{array}
\]
Furthermore, the weighted Einstein equation reads $\rho_f^m=\{(2n+m-2)\lambda+Y_f^m\}g$, so we also have $(\nabla_Y\rho^m_f)(X,Z)=\left(\frac{1}{m}Y^m_f-\lambda \right)df(Y)g(X,Z)$. We can now take the difference $(\nabla_Y\rho^m_f)(X,Z)-(\nabla_Z\rho^m_f)(X,Y)$ to find (cf. \cite{Brozos-Iso-GQE})
\begin{equation}\label{eq:Rnf}
\begin{array}{rcl}
 	R(\nf,X,Y,Z)&=&\left(\frac{1}{m}Y^m_f-\lambda \right)\left\{df(Z)g(X,Y)-df(Y)g(X,Z)\right\} \\
	\noalign{\smallskip}
	&& +(\nabla_Y\rho)(X,Z)-(\nabla_Z\rho)(X,Y) \\
	\noalign{\smallskip}
	&& +\frac{1}{m}\left\{df(Y)\Hes_f(X,Z)-df(Z)\Hes_f(X,Y) \right\}.
\end{array}
\end{equation}
Finally, since $P_f^m=\lambda g$, we have $W_f^m=R-\lambda g\KN g$. Hence
\begin{equation}\label{eq:weighted_divergence_Einstein}
\begin{array}{rcl}
 	\delta_fW_f^m(X,Y,Z)&=&\delta W^m_f(X,Y,Z)-\iota_{\nf} W^m_f(X,Y,Z)\\
 	\noalign{\medskip}
 	&=&\delta R(X,Y,Z)-R(\nf,X,Y,Z) \\
	\noalign{\smallskip}
	&&+2\lambda\{df(Y)g(X,Z)-df(Z)g(X,Y)\}.
\end{array}
\end{equation}
Since $\delta R(X,Y,Z)=(\nabla_Y\rho)(X,Z)-(\nabla_Z\rho)(X,Y)$, a combination of \eqref{eq:Rnf} and \eqref{eq:weighted_divergence_Einstein} yields equation \eqref{eq:Main-Expression-1}.
\end{proof}
We are now ready to prove a first rigidity result, concerning the warped product structure of these SMMSs around regular points of $f$. 

\begin{lemma}\label{lemma:local-splitting}
	Let $(M^n,g,f,m,\mu)$ be a SMMS with $P_f^m=\lambda g$ and $\delta_fW_f^m=0$. Let $p\in M$ be a regular point of $f$. Then, there exists a neighborhood $\mathcal{U}$ of $p$ in $M$ which is isometric to a warped product $I\times_\varphi N$, where $I\subset \mathbb{R}$ is an open interval, $N$ is an $(n-1)$-dimensional Einstein  manifold, and $\nf$ is tangent to $I$.
\end{lemma}
\begin{proof}
 Since $p$ is a regular point of $f$, $\nf\neq 0$ in a neighborhood of $p$. Thus, consider an orthonormal frame $\mathcal{B}=\{E_1,\dots,E_n\}$ around $p$, where  $E_1=\nf/||\nf||$. Since $W_f^m$ is weighted harmonic, the left side of equation~\eqref{eq:Main-Expression-1} vanishes. Consequently, we can take $X=Z=E_1$ and $Y=E_i$, $i\neq 1$, to find 
\begin{equation}\label{eq:Hesian-E1Ei}
\Hes_f(E_1,E_i)=0, \quad  i\neq1,
\end{equation}
which shows that the distribution generated by $\nf$ is totally geodesic. Furthermore, taking $X=E_i$, $Y=E_j$, $Z=E_1$, with $i,j\neq 1$,  equation \eqref{eq:Main-Expression-1} yields
\begin{equation}\label{eq:Hessian-Diagonal}
 	\Hes_f(E_i,E_j)=(Y^m_f+m \lambda)\delta_{ij}, \quad i,j\neq1.
\end{equation}
It follows that the level hypersurfaces of $f$ around $p$ are totally umbilical. Consequently, $(M,g)$ splits in a neighborhood $\mathcal{U}$ of $p$ as a twisted product $I\times_\psi N$, where $I\subset \mathbb{R}$ is an open interval, for some function $\psi$ on $I\times N$ (see \cite{Ponge-Twisted}). 
Moreover, the mean curvature vector field $\nabla f$ is parallel in the normal bundle $\operatorname{span}{\nabla f}$ due to \eqref{eq:Hesian-E1Ei}, so the leaves of the fiber are spherical. Hence, the twisted product reduces to a warped product $I\times_\varphi N$ for some function $\varphi$ on $I$ (see \cite{Hiepko}). Alternatively, since $\rho_f^m=\{(2n+m-2)\lambda+Y_f^m\}g$ and $\Hes_f$ diagonalizes in the frame $\mathcal{B}$, so does the Ricci tensor. Hence, the vanishing condition on mixed terms for $\rho$ given in \cite{Fernandez-Twisted} is satisfied and also implies the reduction of the twisted product to the warped product.

Now we show that $N$ is Einstein as follows. Let $t$ be a coordinate parameterizing $I$ by arc length, and consider the local orthonormal frame $\{\partial_t,E_2,\dots,E_n\}$. Note that $E_2,\dots, E_n$ are tangent to $N$.  Thus, from the weighted Einstein condition and \eqref{eq:Hessian-Diagonal}, we get that 
\[
\begin{array}{rcl}
\rho(E_i,E_j)&=&\rho^m_f (E_i,E_j)-\Hes_f(E_i,E_j)\\
\noalign{\medskip}
&=&\{(2n+m-2)\lambda+Y_f^m\} \delta_{ij}-(Y^m_f+m \lambda)\delta_{ij}\\
\noalign{\medskip}
&=& 2(n-1)\lambda \delta_{ij}
\end{array}
\] 
for $i,j=2,\dots,n$.
Moreover, consider the basis $\{\bar{E_i}=\varphi E_i\}_{i=2,\dots,n}$ which is orthonormal on $N$. From the expression of the Ricci tensor of a warped product (see \cite{Oneill}), we have
\begin{equation}\label{eq:ODE1precursor}
\begin{array}{rcl}
 	\rho^N(\bar{E_i},\bar{E_j})&=&\rho(\bar{E_i},\bar{E_j})+g(\bar{E_i},\bar{E_j})\left(\frac{\varphi''}{\varphi}+(n-2)\frac{(\varphi')^2}{\varphi^2}\right) \\
	\noalign{\medskip}
	&=& \varphi^2 \left(2(n-1)\lambda+\frac{\varphi''}{\varphi}+(n-2)\frac{(\varphi')^2}{\varphi^2}\right)\delta_{ij}.
\end{array}
\end{equation}
	Since $\rho^N(\bar{E_i},\bar{E_j})$ is a function defined on the fiber, it does not depend on $t$, which is a coordinate of the base. Hence,  $\rho^N=\beta g^N$  for some $\beta\in \mathbb{R}$ and $N$ is Einstein.
\end{proof}

\begin{remark}
Note that the local splitting given by Lemma~\ref{lemma:local-splitting} is reminiscent of the result found by Catino \cite{Catino-GQE}, in the unweighted Riemannian setting, for generalized quasi-Einstein manifolds with harmonic Weyl tensor such that $\iota_{\nf}W$=0. However, we do not require both summands in $\delta_fW_f^m$  to vanish, but merely that they cancel out (see Example~\ref{ex:not-Einstein}). 
\end{remark}

By Lemma~\ref{lemma:local-splitting}, whenever we are working locally around any regular point of $f$, we can assume without loss of generality that our SMMSs are built on a warped product of the form $I\times_\varphi N$ with the density function defined on $I$. Such a warped product is not weighted Einstein with weighted harmonic Weyl tensor in general, indeed equation~\eqref{eq:ODE1precursor} imposes a constraint on the warping function. The next result provides necessary and sufficient conditions that identify these  SMMSs in terms of an overdetermined system of ODEs. 

\begin{lemma}\label{lemma:necsuf-ODE}
Let $(I\times_\varphi N,g,f,m,\mu)$ be an $n$-dimensional warped product SMMS where $I$ is an open interval, $\nf$ is tangent to $I$, and such that $\rho^N=\beta g^N$ for some $\beta\in \mathbb{R}$.  
Then $P^m_f=\lambda g$, for $\lambda\in \mathbb{R}$, and $\delta_fW_f^m=0$ if and only if the following system of ODEs is satisfied:
\begin{eqnarray}
 				0&=&\beta-\varphi''\varphi-(n-2)(\varphi')^2-2(n-1)\lambda \varphi^2, \label{eq:warpex1} \\
				\noalign{\medskip}
				0&=&f''-(n-1)\frac{\varphi''}{\varphi}-\frac{1}{m}(f')^2-\frac{\varphi'f'}{\varphi}-2(n-1)\lambda,  \label{eq:densityeq1} \\
				\noalign{\medskip}
				0&=&\frac{\varphi' f'}{\varphi}+(n-m)\lambda-J_f^m \label{eq:hescond1},
	\end{eqnarray}
where the weighted Schouten scalar $J_f^m$ is given by
\[
			\begin{array}{rcll}
 				2(n+m-1)J^m_f&=&(n-1)\frac{\beta-(n-2)(\varphi')^2}{\varphi^2}+2(n-1)\frac{\varphi'f'-\varphi''}{\varphi} \\
				\noalign{\medskip}
				&&+2f''-\frac{1+m}m(f')^2+m(m-1)e^{2f/m}\mu.
			\end{array}
	\]
\end{lemma}
\begin{proof}
Let $t$ be a local coordinate parameterizing $I$ by arc length.
We work in the local orthonormal frame $\mathcal{B}=\{\partial_t,E_2,\dots,E_n\}$. The Bakry-\'Emery Ricci tensor \eqref{eq:Bakry-Emery-Ricci-tensor} takes the form
\[
			\begin{array}{rclll}
 				\rho_f^m(\partial_t,\partial_t)&=&-(n-1)\frac{\varphi''}{\varphi}+f''-\frac{1}{m}(f')^2, \qquad  \rho_f^m(\partial_t,E_i)=0, \\
				\noalign{\medskip}
				\rho_f^m(E_i,E_j)&=&\left(\frac{\beta}{\varphi^2}-\frac{\varphi''}{\varphi}-(n-2)\frac{(\varphi')^2}{\varphi^2}+\frac{\varphi' f'}{\varphi}\right)\delta_{ij}.
			\end{array}
	\]
Thus, the fact that $(I\times_\varphi N,g,f,m,\mu)$ is weighted Einstein, using \eqref{eq:constant-wE} to express it as $\rho_f^m=((n+m-2)\lambda+J_f^m)g$, is equivalent to the following two equations:
\begin{eqnarray}
 				-(n-1)\frac{\varphi''}{\varphi}+f''-\frac{1}{m}(f')^2&=&(n+m-2)\lambda+J_f^m,\label{eq:1} \\			
			\frac{\beta}{\varphi^2}-\frac{\varphi''}{\varphi}-(n-2)\frac{(\varphi')^2}{\varphi^2}+\frac{\varphi' f'}{\varphi}&=&(n+m-2)\lambda+J_f^m.\label{eq:2}
			\end{eqnarray}
On the one hand, a direct calculation on the warped product manifold yields $\Hes_f(E_i,E_i)=\frac{\varphi' f'}{\varphi}$. On the other hand, if $\delta_fW_f^m=0$, then equation \eqref{eq:Hessian-Diagonal} is also satisfied and gives $\Hes_f(E_i,E_i)=(m-n)\lambda+J_f^m$, where we have used that $Y^m_f=J^m_f-n\lambda$ on a weighted Einstein SMMS. Hence, $\frac{\varphi' f'}{\varphi}=(m-n)\lambda+J_f^m$, which is equation~\eqref{eq:hescond1}. Now, substituting the term $J_f^m$ in \eqref{eq:2} and \eqref{eq:1} yields, respectively, \eqref{eq:warpex1} and \eqref{eq:densityeq1}. The form of the weighted Schouten scalar $J_f^m$ follows from a direct computation of the weighted scalar curvature \eqref{eq:Weighted-scalar-curvature}.

Conversely, if \eqref{eq:warpex1}-\eqref{eq:hescond1} are satisfied, then \eqref{eq:1} and \eqref{eq:2} hold and the equation $P_f^m=\lambda g$ is also satisfied. Thus, we only need to check that they are sufficient conditions for the weighted harmonicity condition $\delta_fW_f^m=0$. To that end, we use the expression given by \eqref{eq:Main-Expression-1} for $\delta_fW_f^m$, which applies to any weighted Einstein manifold. By the symmetries of this tensor, we only need to analyze the following terms:
\[
			\begin{array}{rcll}
 				\delta_fW_f^m(\partial_t,E_i,E_j)&=&0, \qquad \delta_fW_f^m(E_i,E_j,E_k)=0, \\
				\noalign{\medskip}
				\delta_fW_f^m(\partial_t,E_i,\partial_t)&=&\frac{f'}{m} \Hes_f(\partial_t,E_i)=0, \\
				\noalign{\medskip}
				\delta_fW_f^m(E_i,\partial_t,E_j)&=& \left(\frac{1}{m}Y^m_f+\lambda \right) f'\delta_{ij}-\frac{1}{m}f' \Hes_f(E_i,E_j)\\
				&=&\frac{1}{m}\left(J^m_f+(m-n)\lambda-\frac{\varphi' f'}{\varphi} \right) f'\delta_{ij}\stackrel{\eqref{eq:hescond1}}{=}0.
			\end{array}
	\]
Hence, equations \eqref{eq:warpex1}-\eqref{eq:hescond1} are sufficient for the warped product $I\times_\varphi N$ to be a weighted Einstein manifold with weighted harmonic Weyl tensor.
\end{proof}

\begin{remark}\label{remark:Ricci}
	 Consider a warped product SMMS $(I\times_\varphi N,g,f,m,\mu)$ as in Lemma~\ref{lemma:necsuf-ODE}, which satisfies \eqref{eq:warpex1}--\eqref{eq:hescond1}. Then, the Ricci tensor is readily determined using equation~\eqref{eq:warpex1} and the expressions of the Ricci tensor for a warped product  (see \cite{Oneill}):
	 \[
	\rho(\partial_t,\partial_t)=-(n-1)\frac{\varphi''}{\varphi},\quad \rho(\partial_t,X)=0, \quad \rho(X,Y)=2(n-1)\lambda g(X,Y),
	 \]
for any $X,Y\in  \mathfrak{X}(N) $.
Therefore, the underlying manifold is Einstein if and only if  $\varphi''=-2\lambda \varphi$. As we will shortly show, this is one of two cases which are allowed for this kind of manifold, with the geometry of the non-Einstein case being very heavily restricted. 
\end{remark}

\begin{lemma}\label{lemma:two-cases}
Let $(I\times_\varphi N,g,f,m,\mu)$ be an $n$-dimensional warped product SMMS where $I$ is an open interval, $\nf$ is tangent to $I$, and such that $\rho^N=\beta g^N$ for some $\beta\in \mathbb{R}$. Let $t$ be a local coordinate parameterizing $I$ by arc length. If $P^m_f=\lambda g$ for some $\lambda\in \mathbb{R}$ and $\delta_fW_f^m=0$, then either $I\times_\varphi N$ is Einstein, or $\varphi(t)=Ae^{-\frac{f(t)}{n-1}}$, for some $A\in \mathbb{R}^+$.
\end{lemma}
\begin{proof}
We adopt the notation in Lemma~\ref{lemma:necsuf-ODE} and keep working in a local orthonormal frame $\mathcal{B}=\{\partial_t,E_2,\dots,E_n\}$. The weighted Einstein condition and the harmonicity of the weighted Weyl tensor guarantee that
\begin{equation}\label{eq:weightdivcomp1}
 	0=\delta_fW_f^m(E_i,\partial_t,E_i)=\delta R(E_i,\partial_t,E_i)-R(\nf,E_i,\partial_t,E_i)+2\lambda f'. \\
\end{equation}
 We will use this to obtain an additional ODE. Firstly, consider the divergence of the Riemann curvature tensor, given by
 $\delta R(X,Y,Z)=(\nabla_Y\rho)(X,Z)-(\nabla_Z\rho)(X,Y)$. On the one hand, $(\nabla_{\partial_t}\rho)(E_i,E_i)=\partial_t(\rho(E_i,E_i))-2\rho(\nabla_{\partial_t}E_i,E_i)$. But $\rho(E_i,E_i)=2(n-1)\lambda$ (see Remark~\ref{remark:Ricci}), so $\partial_t(\rho(E_i,E_i))=0$. Moreover, $g(\nabla_{\partial_t}E_i,E_i)=\frac12\partial_t(g(E_i,E_i))=0$, so $\nabla_{\partial_t}E_i\perp E_i$. Since $\mathcal{B}$ is a basis of eigenvectors for the Ricci operator (see the proof of Lemma~\ref{lemma:local-splitting}), we have $\rho(\nabla_{\partial_t}E_i,E_i)=0$. On the other hand, we use the expression of the connection for a warped product (see \cite{Oneill}) to show that
 \[
 \begin{array}{rcl}
 (\nabla_{E_i}\rho)(\partial_t,E_i)&=&-\rho(\nabla_{E_i} \partial_t,E_i)-\rho(\partial_t,\nabla_{E_i}E_i)\\
\noalign{\medskip}
&=& \frac{\varphi'}{\varphi}\left(\rho(\partial_t,\partial_t)-\rho(E_i,E_i)\right),
\end{array} 
\]
so
\[
 				 \delta R(E_i,\partial_t,E_i)=\frac{\varphi'}{\varphi}(\rho(E_i,E_i)-\rho(\partial_t,\partial_t))=(n-1)\frac{\varphi'(\varphi''+2\lambda\varphi)}{\varphi^2}.
	\]
Additionally, the curvature term $R(\nf, E_i,\partial_t,E_i)$ takes the form $R(\nf, E_i,\partial_t,E_i)=-\frac{\Hes_\varphi(\partial_t,\nf)}{\varphi}=-\frac{\varphi''f'}{\varphi}$. With this, equation \eqref{eq:weightdivcomp1} becomes
\[
 	 \frac{1}{\varphi^2}(\varphi''+2\lambda \varphi)(\varphi f'+(n-1)\varphi')=0.
\]
 Thus, $\varphi''+2\lambda \varphi =0$,  or  $\varphi f'+(n-1)\varphi'=0$. 
 In the first case, the underlying manifold is Einstein (see Remark~\ref{remark:Ricci}). In the  second one, we solve the ODE to get $\varphi(t)=Ae^{-\frac{f(t)}{n-1}}$, for $A\in \mathbb{R}^+$. 
\end{proof}

Lemma~\ref{lemma:two-cases} reduces our study to only two possibilities. We will discuss the Einstein case in detail in Section~\ref{sec:Einstein} and give now the proof of the main local rigidity result.

\bigskip

\noindent{\it Proof of Theorem~\ref{th:main-local-result}.} Let $(M^n,g,f,m,\mu)$ be a SMMS such that $P_f^m=\lambda g$ and $\delta_fW_f^m=0$. By Lemma~\ref{lemma:local-splitting}, around every regular point of $f$ there exists a neighborhood $\mathcal{U}$ which is isometric to a warped product of the form $I\times_\varphi N$, where $I\subset \mathbb{R}$ is an open interval, $\nf$ is tangent to $I$, and $\rho^N=\beta g^N$ for some $\beta\in \mathbb{R}$. Using Lemma~\ref{lemma:two-cases}, we have that either $I\times_\varphi N$ is Einstein and Theorem~\ref{th:main-local-result}--(1) holds, or the warping and density functions are related by $\varphi(t)=Ae^{-\frac{f(t)}{n-1}}$ for some $A\in \mathbb{R}^+$, where $t$ is a coordinate parameterizing $I$ by arc length. Assume that the latter is satisfied. Then, the necessary and sufficient conditions given by Lemma~\ref{lemma:necsuf-ODE} take on a simpler form. Indeed, equation \eqref{eq:warpex1} becomes
\begin{equation}\label{eq:warpex2}
	0=\beta-\frac{A^2}{n-1}e^{\frac{-2f}{n-1}}(2(n-1)^2\lambda+(f')^2-f''),
\end{equation}
while equation \eqref{eq:densityeq1} turns into
\begin{equation}\label{eq:fprimesquared}
	(f')^2=2m(f''-(n-1)\lambda).
\end{equation}
Now, taking the derivative of \eqref{eq:fprimesquared} yields $f^{(3)}=\frac{f'f''}{m}$. Substituting this expression into the derivative of \eqref{eq:warpex2}  and using \eqref{eq:fprimesquared}, we have
\begin{equation}\label{eq:exprder}
	0=\frac{A^2e^{-2f/(n-1)}f'}{m(n-1)^2}(4m(m-n+1)(n-1)\lambda-(4m^2-2mn+n-1)f'').
\end{equation}
Note that the factor $4m^2-2mn+n-1$ vanishes if and only if $m=\frac{1}2$ or $m=\frac{1}2(n-1)$. This results in three cases we need to analyze separately.

\medskip

\noindent\underline{$m\notin\left\{\frac{1}2,\frac{1}2(n-1)\right\}$:} Let $B=\frac{4m(m-n+1)(n-1)\lambda}{4m^2-2mn+n-1}$. Then, from \eqref{eq:exprder} it follows that $f(t)=\frac{B}2t^2+Ct+D$, where $C,D\in \mathbb{R}$. Hence, we have $0=f^{(3)}=\frac{f'f''}{m}$, but $f'\neq 0$, so $f''=B$ must vanish. Since $m>0$ and $n\geq 3$, $B=0$ if and only if $\lambda=0$ or $m=n-1$. 

If $\lambda=0$, from \eqref{eq:fprimesquared}, we have that $f'=0$, which is not possible. Thus, there are no solutions with $m\notin\left\{\frac{1}2,\frac{1}2(n-1)\right\}$ and $\lambda=0$.

If $m=n-1$, then from \eqref{eq:fprimesquared} we deduce $C^2=-2(n-1)^2\lambda$. With this, it follows that $\frac{\varphi''}{\varphi}=-2\lambda$, and the manifold is Einstein (see Remark~\ref{remark:Ricci}).

\medskip

\noindent\underline{$m=\frac{1}2(n-1)$:} From \eqref{eq:exprder}, it follows that $\lambda=0$. Solving \eqref{eq:fprimesquared}, and through a suitable change of the coordinate $t$ (preserving the parameterization by arc length), if needed, we find $f(t)=-(n-1)\log(E t)$, where $E\in \mathbb{R}^+$. Hence, $\varphi(t)=AEt$, so $\varphi''=0$ and the manifold is Einstein.

\medskip

\noindent\underline{$m=\frac{1}2$:} From \eqref{eq:exprder}, it follows that $(2n-3)\lambda=0$ and, since $n\geq 3$, we get that $\lambda=0$. Now, we solve \eqref{eq:fprimesquared} (translating $t$ if needed) to find that $f(t)=-\log(Bt)$, where $B\in \mathbb{R}^+$, and hence $\varphi(t)=A(Bt)^{\frac{1}{n-1}}$. Then, \eqref{eq:warpex2} reduces to $\beta=0$, so the fiber $N$ must be Ricci flat. Finally, a direct computation shows that the condition given by \eqref{eq:hescond1} reduces to $0=\frac{\mu}{4(2n-1)(Bt)^4}$, so $\mu=0$. Hence, a  SMMS of the form $(I\times_\varphi N,g,f,\frac{1}2,0)$ satisfies equations \eqref{eq:warpex1}, \eqref{eq:densityeq1} and \eqref{eq:hescond1} and, moreover, is the only solution whose underlying manifold is not Einstein. This corresponds to Theorem~\ref{th:main-local-result}--(2). \qed

\begin{remark}\label{remark:Einstein-connected-component}
	Note that, by the real analyticity of the metric (see Theorem~\ref{th:analytic_solutions}), if $(M,g)$ is Einstein in an open set, then it is Einstein in the whole connected component containing it. Thus, since we are dealing with connected manifolds, the Einstein behavior around a single regular point is enough to infer that the whole manifold is Einstein with $\rho=2(n-1)\lambda g$, even if the parameters $m$, $\lambda$ and $\mu$ were to coincide with those acceptable in the non-Einstein case given by Example~\ref{ex:not-Einstein}.
\end{remark}

\section{The Einstein case}\label{sec:Einstein}

	Let $(M^n,g,f,m,\mu)$ be a SMMS with $P_f^m=\lambda g$ and $\delta_fW_f^m=0$. By Theorem~\ref{th:main-local-result}, around any regular point $p$ of $f$, $M$ is isometric to a warped product $I\times_\varphi N$, where $N$ is Einstein. We have already shown that the non-Einstein case is heavily restricted, with only one allowed value for the parameters $\lambda$, $m$ and $\mu$ and for the Einstein constant of the fiber $\beta$, with the warping and density functions also fixed up to integration constants. However, the next result shows that if the total space $I\times_\varphi N$ is Einstein (which implies that the connected component of $M$ containing $p$ is Einstein by analyticity), its geometry is more flexible, allowing for solutions to the necessary and sufficient equations \eqref{eq:warpex1}-\eqref{eq:hescond1} for different combinations of parameters and functions. The value $m=1$ is exceptional and is excluded in the statement, although this case also follows with an extra degree of freedom (see Remark~\ref{remark:m=1} below).
	

\begin{theorem}\label{th:Einstein-case} 
	Let $(M^n,g,f,m,\mu)$ be a SMMS with $(M,g)$ Einstein and such that $P_f^m=\lambda g$ and $\delta_fW_f^m=0$, with $m\neq 1$. Then, for each regular point $p$ of $f$, there exists a neighborhood $\mathcal{U}$ of $p$ which is isometric to a warped product $I\times_\varphi N$, where $I\subset \mathbb{R}$ is an open interval, $\rho^N=\beta g^N$, and $f$, $\varphi$, $\beta$ and $\mu$ take the following forms ($t$ is a coordinate parameterizing $I$ by arc length):\black
	\begin{enumerate}
	\item{If $\lambda>0$, then
		\[
		\begin{array}{rcl}
			\varphi (t)&=&c_1\cos(\sqrt{2\lambda}t)+c_2\sin(\sqrt{2\lambda}t), \\
			\noalign{\medskip}
			f(t)&=&-m\log\left(c_3+c_2c_4(\cos(\sqrt{2\lambda}t)-1)-c_1c_4\sin(\sqrt{2\lambda}t)\right),		
		\end{array}
		\]
		and
		\[
			\beta=2 (c_1^2+c_2^2)(n-2)\lambda, \qquad \mu=2(c_1^2c_4^2+2c_2c_3c_4-c_3^2)\lambda,
		\]
		where $c_1,c_3\in \mathbb{R}^+$, $c_2\in \mathbb{R}$, $c_4\in \mathbb{R}^*$.}
	\medskip
	\item{If $\lambda=0$, then
		\[
		\begin{array}{rcl}
			\varphi (t)&=&c_1t+c_2, \\
			\noalign{\medskip}
			f (t)&=&-m\log\left(c_3-c_4(c_1t^2+2c_2t)\right),
		\end{array}
		\]
		and
		\[
			\beta=c_1^2(n-2), \qquad \mu=4c_4(c_1c_3+c_2^2c_4),
		\]
		where $c_1\in\mathbb{R}$, $c_2,c_3\in \mathbb{R}^+$, $c_4\in \mathbb{R}^*$.}
	\medskip
		\item{If $\lambda<0$, then
		\[
		\begin{array}{rcl}
			\varphi (t)&=&c_1e^{\sqrt{-2\lambda}\,t}+c_2e^{-\sqrt{-2\lambda}\,t}, \\
			\noalign{\medskip}
			f(t)&=&-m\log\left(c_3+c_2c_4(e^{-\sqrt{-2\lambda}t}-1)-c_1c_4(e^{\sqrt{-2\lambda}t}-1)\right),
		\end{array}
		\]
		and
		\[
		\beta=8 c_1 c_2(n-2)\lambda, \qquad \mu=-2(2(c_1-c_2)c_3c_4+(c_1+c_2)^2c_4^2+c_3^2)\lambda,
		\]
		where $c_1+c_2\in \mathbb{R}^+$, $c_3\in\mathbb{R}^+$, $c_4\in\mathbb{R}^*$.}
\end{enumerate}
\end{theorem}
\begin{proof}
	Take a regular point of $p$ of $f$, and the local splitting as a warped product $I\times_\varphi N$ given by Theorem~\ref{th:main-local-result} around $p$. The weighted Einstein condition and the harmonicity of the weighted Weyl tensor guarantee that the necessary and sufficient equations \eqref{eq:warpex1}-\eqref{eq:hescond1} are satisfied. Moreover, since the underlying manifold is Einstein, the ODE $\varphi''+2\lambda \varphi=0$ must also be satisfied (see Remark~\ref{remark:Ricci}). Note that, by a translation, if needed, we can adjust the coordinate $t$ so that $0\in I$. Solving this equation, we can fix the different forms of $\varphi$, depending on the sign of $\lambda$.  Thus, for $\lambda>0$, $\varphi(t)=c_1\cos(\sqrt{2\lambda}t)+c_2\sin(\sqrt{2\lambda}t)$, where the constants $c_1,c_2$ are given by the data of the corresponding initial value problem. Substituting $\varphi$ into equation \eqref{eq:warpex1} yields
	\[
	0=\beta-2 (c_1^2 + c_2^2) (n-2) \lambda,
	\]
	from where $\beta=2 (c_1^2 + c_2^2) (n-2) \lambda$.
	Now, equation \eqref{eq:densityeq1} imposes that
\[
\begin{array}{rcl}	\small
	0&=&m f''(t) \left(c_2 \sin
		(\sqrt{2\lambda }
		t)+c_1 \cos
		(\sqrt{2\lambda }
		t)\right)\\
		\noalign{\smallskip}
	&&	-\sqrt{2\lambda }\, m f'(t)
		\left(c_2 \cos
		(\sqrt{2\lambda }
		t)-c_1 \sin
		(\sqrt{2\lambda }
		t)\right)\\
		\noalign{\smallskip}
	&&	-f'(t)^2
		\left(c_2 \sin
		(\sqrt{2\lambda }
		t)+c_1 \cos
		(\sqrt{2\lambda }
		t)\right),
	\end{array}
	\]
	from where the form of $f$ is obtained for integration constants $c_3>0$ and $c_4\neq 0$. 

	Lastly, in order to fix the value of $\mu$, we substitute the known values of $\beta$, $\varphi$ and $f$ into the last of the necessary and sufficient equations, which is \eqref{eq:hescond1}:
	\[
		0=\frac{m(m-1)(2(c_1^2c_4^2+2c_2c_3c_4-c_3^2)\lambda-\mu)}{2(n+m-1)(c_1c_4\sin(\sqrt{2\lambda}t)-c_2c_4\cos(\sqrt{2\lambda}t)+c_2c_4-c_3)^2}.
	\]
From this expression we get that $\mu=2(c_1^2c_4^2+2c_2c_3c_4-c_3^2)\lambda$, which concludes the case $\lambda>0$. If $\lambda=0$ or $\lambda<0$ the argument is analogous and we omit details.
\end{proof}

\begin{remark}\label{remark:m=1}
	In the case $m=1$, Theorem~\ref{th:Einstein-case} still holds for the same values of $\varphi$, $f$ and $\beta$. The only difference with the case $m\neq 1$ is that, since the auxiliary curvature parameter $\mu$ does not appear in the definition of any weighted tensors when $m=1$, equations \eqref{eq:warpex1}-\eqref{eq:hescond1}  are satisfied for arbitrary values of $\mu$.
\end{remark}


Note that, for Einstein manifolds, the Riemann curvature tensor decomposes as $R=\frac{\tau}{2n(n-1)}g\KN g+W$. If $(M^n,g,f,m,\mu)$ is an Einstein SMMS with $P_f^m=\lambda g$ and $\delta_fW_f^m=0$, then $\rho=2(n-1)\lambda g$ and $\tau=2n(n-1)\lambda$ (see Remark~\ref{remark:Ricci}), which implies that $W=W^m_f=R-\lambda g\KN g$, i.e. the weighted and unweighted Weyl tensors become equal. Consequently, the following three conditions are equivalent in this context:
\begin{enumerate}
	\item{$M$ has constant sectional curvature.}
	\item{$M$ is locally conformally flat in the usual sense.}
	\item{$M$ is locally conformally flat in the weighted sense.}
\end{enumerate}

Moreover, a warped product $I\times_\varphi N$ is locally conformally flat if and only if $N$ has constant sectional curvature (see \cite{Brozos-Loc-Conf-Flat}). Then, the following rigidity result in low dimensions follows immediately. 

\begin{corollary}\label{cor:4-dim-curv}
	Let $(M^{3,4},g,f,m,\mu)$ be an Einstein SMMS such that $P_f^m=\lambda g$ and $\delta_fW_f^m=0$. Then, $(M,g)$ has constant sectional curvature $2\lambda$.
\end{corollary}
\begin{proof}
By Theorem~\ref{th:main-local-result}, the open dense set $\tilde{M}\subset M$ of regular points of $f$ is locally isometric to a warped product $I\times_\varphi N$ with the fiber $N$ a $2$ or $3$-dimensional Einstein manifold. Hence, $N$ has constant sectional curvature and $\tilde{M}$ is locally conformally flat (see \cite{Brozos-Loc-Conf-Flat}). By the smoothness of the Weyl and Cotton tensors, it follows that $M$ is locally conformally flat. 
\end{proof}

%

	Nevertheless, for $n\geq 5$, Corollary~\ref{cor:4-dim-curv} no longer holds, and there exist Einstein SMMSs which are weighted Einstein and have weighted harmonic Weyl tensor, but are not locally conformally flat. In order to build an example, it suffices to consider a warped product $I\times_\varphi N$ with $N$ Einstein but not locally conformally flat. The following construction illustrates this fact.

\begin{example}\rm
	Let $(M,g)$ be the warped product $I\times_\varphi N$, where $N$ is the Riemannian product $S_1\times S_2$ of two surfaces of constant Gauss curvature $\beta$. Thus, $N$ is Einstein with $\rho^N=\beta g^N$. Choose  local coordinates $(x_1,x_2)$ and $(x_3,x_4)$, respectively, for $S_1$ and $S_2$ and consider the metric of the warped product given by the non-vanishing components
		\[
			\begin{array}{rcll}
 				g(\partial_t,\partial_t)=1, \qquad g(\partial_{x_1},\partial_{x_1})=g(\partial_{x_2},\partial_{x_2})&=&\frac{\varphi(t)^2}{\left(1+\frac{\beta}{4}(x_1^2+x_2^2)\right)^2}, \\
				\noalign{\medskip}
				g(\partial_{x_3},\partial_{x_3})=g(\partial_{x_4},\partial_{x_4})&=&\frac{\varphi(t)^2}{\left(1+\frac{\beta}{4}(x_3^2+x_4^2)\right)^2}.
			\end{array}
	\]
	Now, for $\lambda\in \mathbb{R}$, fix $\varphi(t)$, $f(t)$, $\beta$ and $\mu$ as in Theorem~\ref{th:Einstein-case} (in agreement with the sign of $\lambda$), choosing constants such that $\beta\neq 0$. The SMMS defined by $(I\times_\varphi N,g, f,m,\mu)$ is Einstein, and satisfies $P_m^f=\lambda g$ and $\delta_fW_f^m=0$, but is not of constant sectional curvature. Indeed, up to symmetries, the nonzero components of the usual Weyl tensor (hence also of the weighted Weyl tensor) are
		\[
			\begin{array}{rcl}
 				W(\partial_{x_1},\partial_{x_2},\partial_{x_1},\partial_{x_2})&=&\tfrac{512\, \beta \, \varphi(t)^2}{3\left(4+\beta(x_1^2+x_2^2)\right)^4}, \,\, W(\partial_{x_3},\partial_{x_4},\partial_{x_3},\partial_{x_4})=\tfrac{512\, \beta \, \varphi(t)^2}{3\left(4+\beta(x_3^2+x_4^2)\right)^4}, \\
				\noalign{\medskip}
				W(\partial_{x_i},\partial_{x_j},\partial_{x_i},\partial_{x_j})&=&-\tfrac{256\, \beta \, \varphi(t)^2}{3\left(4+\beta(x_1^2+x_2^2)\right)^2\left(4+\beta(x_3^2+x_4^2)\right)^2}, \quad i=1,2,\, j=3,4.
			\end{array}
	\]
\end{example}


\begin{remark}\label{remark:multi-warped}
Let us consider a SMMS $(M^n,g,f,m,\mu)$ which is Einstein, with $P_f^m=\lambda g$ and $\delta_fW_f^m=0$, and adopt the notation in Theorem~\ref{th:Einstein-case}. Note that, if we make the change of variable $v=e^{-f/m}$ in Theorem~\ref{th:Einstein-case}, we find that the density and warping functions satisfy $v'(t)=A^{-1}\varphi(t)$, where $A\neq 0$ is an integration constant fixed by the initial data.
Thus, the warped product $I\times_{Av'}N$ is Einstein and the conformal metric $ v^{-2} g$ is also Einstein, since $v$ is a solution of $\Hes_v-\frac{1}n \Delta v g=0$ (see \cite{Kuhnel-conformal}). A classical result by Brinkmann \cite{Brinkmann} states that warped products in Theorem~\ref{th:Einstein-case} are characteristic of Einstein metrics that are conformally transformed into Einstein metrics.

Moreover, these warped products have harmonic Weyl tensor in the usual sense (see Remark~\ref{remark:harmonic-Weyl}). In fact, we have that $W_f^m=W$, so the harmonicity condition $\delta_fW_f^m=0$ can be reformulated in terms of $W$ as the condition $\iota_{\nf}W=0$. Also, notice that the divergence of the Weyl tensor is modified by a conformal change $\tilde{g}=e^{-2f}g$ as $\tilde{\delta} \tilde{W}=\delta{W}+(3-n)\iota_{\nf} W$ (see \cite{Kuhnel-conformal}). Hence, given that $W$ is harmonic and  $\iota_{\nf}W=0$, it follows that $\tilde{\delta} \tilde{W}=0$, so we can rephrase the role of the density function by stating that {\it it defines a conformal change of the metric that preserves the harmonicity of the Weyl tensor of $M$}.
	
In terms of the geometric interpretation of weighted objects as Riemannian invariants of the formal warped product  \eqref{eq:formal-warped-product}, we have that its auxiliary manifold becomes a multiply warped product of the form $I\times_{Av'}N\times_{v}F^m(\mu)$. Multiply warped product metrics have been considered in different contexts to obtain examples of manifolds with some curvature features; see, for example, \cite{pacific,Dobarro-Unal} for studies of these metrics related to the Einstein condition, local conformal flatness and negative curvature. Similarly to how quasi-Einstein manifolds can be used as bases to find Einstein warped products (see \cite{Kim-Kim-Warped}), we can use weighted Einstein manifolds with weighted harmonic Weyl tensor in order to find multiply warped products satisfying certain geometric properties. For example, these multiply warped products  have harmonic Weyl tensor. Indeed, notice that a conformal change of the form $g\oplus v^2 g^F\mapsto \frac{1}{v^2} g\oplus g^F$ transforms the warped product into a direct product of two Einstein manifolds, so this product manifold has harmonic Weyl tensor. Since $v$ only depends on $t$ and $\iota_{\partial_t} W=0$, the inverse of the previous conformal change preserves the vanishing of the divergence of the Weyl tensor. This harmonicity can also be proved directly on the multiply warped product by checking the conditions on \cite{Gebarowski}.		
			
Furthermore, although the multiply warped products $I\times_{Av'}N\times_{v}F^m(\mu)$ are not Einstein in general, we can make them so by choosing appropriate constants of integration, namely $c_3=c_2c_4$ (if $\lambda>0$), $c_3=c_4(c_2-c_1)$ (if $\lambda<0$) and $c_1=0$ (if $\lambda=0$), with the Einstein constant being $\tilde{\lambda}=2(n+m-1)\lambda$. A direct computation shows that this is equivalent to taking the scale of the SMMS to be zero, thus making it a quasi-Einstein manifold (see Lemma~\ref{lemma:scale}).
		\end{remark}


\section{Global results}\label{sec:global-results}
Now we turn our attention to global questions and study obstructions to the existence of complete weighted Einstein manifolds with weighted harmonic Weyl tensor. In related contexts, several authors have given results for complete and simply connected quasi-Einstein manifolds (see, for example, \cite[Theorem 1.2.]{He-warped-Einstein}). However, in this weighted setting, we will show that taking advantage of a relation between the weighted Einstein and generalized Obata equations (see \cite{Obata-Sphere,Wu-Ye-Obata}), we can disregard simple connectedness and prove Theorem~\ref{th:complete-global} by imposing only the completeness assumption.

The following two lemmas highlight some properties of both the Einstein and non-Einstein cases which will be key in our proof of Theorem~\ref{th:complete-global}.

\begin{lemma}\label{lemma:Ricci-blowup}
	Let $(I\times_\varphi N,g,f,\frac{1}2,0)$ be a SMMS given as in Example~\ref{ex:not-Einstein}. Then,  $(I\times_\varphi N,g)$ is incomplete, and cannot be isometrically embedded in any complete manifold.
\end{lemma}
\begin{proof}

The Ricci operator of $(I\times_\varphi N,g)$ has only one non-zero component:
\[
Ric(\partial_t)=\frac{(n-2)}{(n-1)t^2}\partial_t.
\]
Moreover, $\alpha(t)=t$ is a geodesic, since $\alpha'(t)=\partial_t$ and $\nabla_{\alpha'(t)} \alpha'(t)=\nabla_{\partial_t} \partial_t=0$. Notice that $\rho(\alpha'(t),\alpha'(t))=\frac{(n-2)}{(n-1)t^2}$ on $I$. If the manifold were complete, extending the geodesic would yield $\rho(\alpha'(t),\alpha'(t))=\frac{(n-2)}{(n-1)t^2}$ for $t\in (0,\infty)$. Now, note that
\[
\lim_{t\to0^+} \rho(\alpha'(t),\alpha'(t))=\infty,
\] 
so the manifold exhibits Ricci blowup as $t\rightarrow 0^+$. The result follows. 
\end{proof}

\begin{lemma}\label{eq:weighted-Einstein-Obata}
	Let $(M^n,g,f,m,\mu)$ be a SMMS with $(M,g)$ Einstein and such that $P_f^m=\lambda g$ and $\delta_fW_f^m=0$. Then, the function $v=e^{-\frac{f}m}$ is a solution in $M$ of the generalized Obata equation
	\begin{equation}\label{eq:generalized-Obata-equation}
		\Hes_{v}+f(v)g=0,
	\end{equation}
	with $f(v)=2\lambda v-\kappa$, where $\kappa\in \mathbb{R}$ is the scale of $(M^n,g,f,m,\mu)$.
\end{lemma}
\begin{proof}
Since $(M,g)$ is Einstein, by Theorem~\ref{th:main-local-result} we know that $\rho=2(n-1)\lambda g$. Now, by the change of variable $v=e^{-\frac{f}m}$, we have $\Hes_f-\frac{1}m df\otimes df=-\frac{m}v\Hes_v$. Using the scale equation $J_f^m=(m+n)\lambda-m\kappa e^{\frac{f}m}$ (see Lemma~\ref{lemma:scale}), the weighted Einstein equation  $P_f^m=\lambda g$ reads 
\[
\begin{array}{rcl}
\lambda g&=&\frac{1}{n+m-2}\left\{\rho+\Hes_f-\frac{1}m df\otimes df-J_f^m\right\}\\
&=&\frac{1}{n+m-2}\left(-\frac{m}v\Hes_v +((n-m-2)\lambda+\frac{m}v \kappa)g\right),
\end{array}
\]	
from where
\[
 				-m\Hes_v+(m\kappa-2m\lambda v) g=0,
\]
and the result follows.
\end{proof}

Global solutions to the generalized Obata equation have been studied in \cite{Wu-Ye-Obata}. We recall the needed construction from that reference as follows.	Let $f$ be a smooth function defined on an interval $I=(a,b)$, $[a,b)$, $(a,b]$ or $[a,b]$ (with $a$, $b$ possibly infinite), and assume that there exists $\mu\in I$ such that $f(\mu)\neq 0$. Let $u$ be the unique maximally extended solution of the initial value problem
	\[
		u''+f(u)=0, \quad u(0)=\mu, \quad u'(0)=0.
	\]
	Let $T$ be the supremum of $t$ such that $u$ is defined on $[0,t]$, and define the following warped metric on $(0,T)\times S^{n-1}$:
	\begin{equation}\label{eq:mfmu}
		g=dt^2+f(\mu)^{-2}(u')^2g_{S^{n-1}},
	\end{equation}
which extends smoothly through $t=0$ to the Euclidean open ball  $\mathcal{B}_T(0)$. Note that, if $T=\infty$, then $\mathcal{B}_T(0)=\mathbb{R}^n$.  On the other hand, if $T$ is finite, $g$ extends smoothly to $S^n$, where $S^n\setminus\{p,-p\}$ is identified with $(0,T)\times S^{n-1}$. These extensions are denoted by $M_{f,\mu}$. Moreover, taking $v=u(t)$ on $(0,T)\times S^{n-1}$ guarantees that $v$ extends smoothly to $M_{f,\mu}$ (note that $v$ has a critical point at $t=0$ and $v(0)=\mu$) and satisfies the generalized Obata equation \eqref{eq:generalized-Obata-equation}.

\begin{theorem}\cite[Theorem 4.6]{Wu-Ye-Obata} \label{th:generalized-Obata-equation}
	Let $(M^n,g)$ be a connected complete Riemannian manifold admitting a non-constant smooth solution $v$ of the generalized Obata equation \eqref{eq:generalized-Obata-equation} for a smooth function $f$. If $v$ has critical points (at most, it can have two), then $(M,g)$ is isometric to some manifold $M_{f,\mu}$. Otherwise, $(M,g)$ is isometric to a warped product $\mathbb{R}\times_\varphi N$, where $N$ is connected and complete and $v$ is defined on the base of the product.
\end{theorem}

Now, using the previous results, we are ready to prove the global result characterizing complete Einstein SMMSs with weighted harmonic Weyl tensor.
\medskip

\noindent{\it Proof of Theorem~\ref{th:complete-global}.} Let $(M^n,g,f,m,\mu)$ be a complete SMMS such that $P_f^m=\lambda g$ and $\delta_fW_f^m=0$. By Lemma~\ref{th:main-local-result}, around regular points of $f$, $(M,g)$ is either Einstein or given by Example~\ref{ex:not-Einstein}. If $(M,g)$ is Einstein around any regular point of $f$, then it is Einstein everywhere by analyticity (see Remark~\ref{remark:Einstein-connected-component}). On the other hand, if $(M,g)$ is not Einstein, Lemma~\ref{lemma:Ricci-blowup} guarantees that  Example~\ref{ex:not-Einstein} cannot be isometrically embedded in a complete manifold, so $(M,g)$ cannot be complete. Thus, we  assume that $(M,g)$ is Einstein henceforth. Then, by Lemma~\ref{eq:weighted-Einstein-Obata}, $v=e^{-\frac{f}m}$ satisfies the generalized Obata equation \eqref{eq:generalized-Obata-equation} with $f(v)=2\lambda v-\kappa$.

Firstly, assume that $v$ has critical points. Then, by Theorem~\ref{th:generalized-Obata-equation}, $(M,g)$ is isometric to an $M_{f,\mu}$. Thus, the metric becomes
\[
	g=dt^2+\varphi(t)^2 g_{S^{n-1}}, \quad t\in (0,T),
\] 
with $\varphi(t)=\frac{(v'(t))}{(2\lambda v(0)-\kappa)}$, where 
\[
		v''+2\lambda v-\kappa=0, \quad v(0)=\xi>0, \quad v'(0)=0.
\]
If $\lambda>0$, it follows that $v(t)=\frac{\kappa}{2\lambda}+\frac{(2 \xi\lambda-\kappa)\cos(\sqrt{2\lambda}t)}{2\lambda}$, with $t\in (0,\frac{\pi}{\sqrt{2\lambda}})$. The warping function is $\varphi(t)=\frac{\sin \left(\sqrt{2\lambda } t\right)}{\sqrt{2\lambda }}$. Hence, $(M,g)$ is isometric to a sphere of constant sectional curvature $2\lambda$. By imposing the weighted Einstein equation, we obtain that $\mu=2 \xi (\xi \lambda -\kappa )$ or  $m=1$  and, therefore, $(M,g,f,m,\mu)$ and $(M,g,f,1)$ are identified with the $m$-weighted $n$-sphere as in Example~\ref{ex:sphere}. This is Theorem~\ref{th:complete-global}--(1). On the other hand, if $\lambda=0$, $\kappa\neq 0$, then $v(t)=\xi+\frac{\kappa}2 t^2$ with $t\in (0,\infty)$; the warping function is $\varphi(t)=t$ and $\mu=-2 \xi\kappa$, so $(M,g,f,m,\mu)$ is identified with the $m$-weighted $n$-Euclidean space as in Example~\ref{ex:euclidean-space}. This is Theorem~\ref{th:complete-global}--(2). Finally,  if $\lambda<0$, we have $v(t)=\frac{\kappa }{2 \lambda }+\frac{(2 \xi \lambda -\kappa) \cosh \left(\sqrt{-2\lambda }t\right)}{2 \lambda }$, with $t\in (0,\infty)$. The warping function is given by $\varphi(t)=\frac{\sinh \left(\sqrt{-2\lambda } t\right)}{\sqrt{-2\lambda }}$. Similarly to the first case, the weighted Einstein equation yields $\mu=2 \xi (\xi \lambda -\kappa )$ and $(M,g,f,m,\mu)$ is identified with the $m$-weighted $n$-hyperbolic space as in Example~\ref{ex:hyperbolic-space}. This is Theorem~\ref{th:complete-global}--(3.a).

Lastly, consider all remaining cases, where it is assumed that $v$ has no critical points. Then, by Theorem~\ref{th:generalized-Obata-equation}, $(M,g)$ splits globally as a warped product $\mathbb{R}\times_\varphi N$ where $N$ is complete. Thus, the forms of the warping and density functions given by Theorem~\ref{th:Einstein-case} for the Einstein case can be taken to be global. For $\lambda \geq 0$, these density functions either present critical points (so they correspond to a local description of one of the previous examples) or are such that $v=e^{-\frac{f}m}$ turns nonpositive for some values of a coordinate $t$ parameterizing $\mathbb{R}$ by arc length (so they result in incomplete manifolds). Hence, let us focus on the case $\lambda<0$. The form of the warping function is $\varphi (t)=c_1e^{\sqrt{-2\lambda}\,t}+c_2e^{-\sqrt{-2\lambda}\,t}$ as in Theorem~\ref{th:Einstein-case}--(2). For $\varphi$ to stay positive for all $t\in\mathbb{R}$, $c_1$ and $c_2$ must be nonnegative. Note that this also prevents $v(t)=c_3+c_2c_4(e^{-\sqrt{-2\lambda}t}-1)-c_1c_4(e^{\sqrt{-2\lambda}t}-1)$ from presenting critical points. In addition, $v$ must remain positive for all $t\in \mathbb{R}$. Assume first that $c_1,c_2>0$, then $v$ turns nonpositive for large enough values of $t$ if $c_4>0$, and for small enough values of $t$ if $c_4<0$, so this case is not admissible. Hence, either $c_1>0$ and $c_2=0$, or $c_1=0$ and $c_2>0$. Notice that a reparametrization of the form $t\to-t$ together with a change $c_4\to -c_4$ interchange $c_1$ and $c_2$, so we can assume $c_2=0$ and $v(t)=c_3-c_1c_4(e^{\sqrt{-2\lambda}t}-1)$.   Thus, $v$ remains positive if and only if $c_4<0$ and $-c_1c_4\leq c_3$.
	
	Now, let $A=c_1$, $B=c_3$ and $C=-c_4$. It follows from Theorem~\ref{th:Einstein-case}--(2) that $\beta=0$ (hence $N$ is Ricci flat),
	\[
	 	\varphi (t)=A e^{\sqrt{-2\lambda}\,t}, \qquad f(t)=-m\log\left(B+ AC( e^{\sqrt{-2\lambda}\,t}-1)\right),
	 \]
	 and, moreover, either $m=1$ (see Remark~\ref{remark:m=1}) or $\mu=-2(B-AC)^2\lambda$, with $A,B,C\in \mathbb{R}^+$ and $AC\leq B$. This is the remaining case, Theorem~\ref{th:complete-global}--(2.b).
	 \qed
\begin{remark}\rm
Notice that, if $\lambda\geq 0$, the weighted space forms in Examples~\ref{ex:sphere} and \ref{ex:euclidean-space} are the only complete SMMSs which are weighted Einstein and have weighted harmonic Weyl tensor. In contrast, if $\lambda<0$, there are two families of examples, namely those given in Theorem~\ref{th:complete-global}--(3.a) and (3.b). If the dimension is $n\leq 4$, both of these latter examples have an underlying manifold of negative constant sectional curvature (see Corollary~\ref{cor:4-dim-curv}). Nevertheless, these two SMMSs are not identified with each other, indeed the density function has one critical point in the $m$-weighted $n$-hyperbolic space, but has no critical points in Theorem~\ref{th:complete-global}--(3.b).   

For $n\geq 5$, any complete Ricci flat manifold (non-flat) $N$ give rise to a complete SMMS with the construction in Theorem~\ref{th:complete-global}--(3.b). Moreover, the underlying Riemannian manifold does not have constant sectional curvature.
\end{remark}

\bigskip

%


\end{document}